\documentclass[a4paper,12pt, reqno]{amsart}

\usepackage[margin=3.5cm]{geometry}

\usepackage{amsfonts}
\usepackage{amssymb}
\usepackage[colorlinks, linkcolor=blue, anchorcolor=blue, citecolor=blue]{hyperref}

\newtheorem{theorem}{Theorem}[section]
\newtheorem{corollary}[theorem]{Corollary}
\newtheorem{lemma}[theorem]{Lemma}
\newtheorem{proposition}[theorem]{Proposition}

\theoremstyle{definition}

\theoremstyle{remark}
\newtheorem{remark}[theorem]{Remark}

\numberwithin{equation}{section}

\allowdisplaybreaks

\begin{document}

\title[A new class of Carleson measures and integral operators]
{A new class of Carleson measures and integral operators on Bergman spaces}

\author[Hicham Arroussi]{Hicham Arroussi}
\address{\newline Hicham Arroussi \newline Department of Mathematics and Statistics, University of Helsinki, P.O.Box 68, Helsinki, Finland
\newline Department of Mathematics and Statistics, University of Reading, England}
\email{arroussihicham@yahoo.fr, h.arroussi@reading.ac.uk, \newline hicham.arroussi@helsinki.fi}

\author[Huijie Liu]{Huijie Liu}
\address{\newline Huijie Liu \newline Department of Mathematics, Hebei University of Technology, Tianjin 300401, China}
\email{liuhuijie\_l@163.com}

\author[Cezhong Tong]{Cezhong Tong}
\address{\newline Cezhong Tong \newline Department of Mathematics, Hebei University of Technology, Tianjin 300401, China}
\email{ctong@hebut.edu.cn, cezhongtong@hotmail.com}

\author[Zicong Yang]{Zicong Yang}
\address{\newline Zicong Yang \newline Department of Mathematics, Hebei University of Technology, Tianjin 300401, China.}
\email{zc25@hebut.edu.cn}

\begin{abstract}
Let $n$ be a positive integer and $\mathbf{g}=(g_0,g_1,\cdots,g_{n-1})$, with $g_k\in H(\mathbb{D})$ for $k=0,1,\cdots,n-1$. Let $I_{\mathbf{g}}^{(n)}$ be the generalized Volterra-type operators on $H(\mathbb{C})$, which is represented as
$$
I_{\mathbf{g}}^{(n)}f=I^n\left(fg_0+f'g_1+\cdots+f^{(n-1)}g_{n-1}\right),
$$
where $I$ denotes the integration operator
$$(If)(z)=\int_0^zf(w)dw,$$
and $I^n$ is the $n$th iteration of $I$. This operator is a generalization of the operator that was introduced by Chalmoukis in \cite{Cn}. In this paper, we study the boundedness and compactness of the operator $I_{\mathbf{g}}^{(n)}$ acting on Bergman spaces to another. As a consequence of these characterizations, we obtain conditions for certain linear differential equations to have solutions in Bergman spaces. Moreover, we study the boundedness, compactness and Hilbert-Schmidtness of the following sums of generalized weighted composition operators: Let $\mathbf{u}=(u_0,u_1,\cdots,u_n)$ with $u_k\in H(\mathbb{D})$ for $0\leq k\leq n$ and $\varphi$ be an analytic self-map of $\mathbb{D}.$ The sums of generalized weighted composition operators is defined by 
$$L_{\mathbf{u},\varphi}^{(n)}=\sum_{k=0}^nW_{u_k,\varphi}^{(k)},$$
where 
$$W_{u_k,\varphi}^{(k)}f=u_k\cdot f^{(k)}\circ\varphi.$$
Our approach involves the study of  new class of Sobolev-Carleson measures for classical Bergman spaces on unit disk  which appears in the first main Theorems \ref{Theorem1.1} and \ref{Theorem1.2}.
\end{abstract}

\subjclass[2020]{30H20; 47B38; 46E35.}
\keywords{Bergman space; Sobolev Carleson measure; Volterra type operator; composition-differentiation operator.}

\maketitle

\section{Introduction and main results}

\subsection{Carleson measure} 

Carleson measures are essential in modern complex analysis, harmonic analysis and operator theory. The notion of Carleson measures for the Bergman spaces is first introduced by Hastings in \cite{Hw} and further pursued by Luecking in \cite{Ld,Ld1} and others. In the study of linear operators on spaces of holomorphic functions, the Carleson measure is a powerful tool to characterize the properties of many classical operators such as Toeplitz operators, Hankel operators, Volterra type operators and composition operators. Our motivation comes from the study of integral operators on the Bergman space along the line: Aleman and Siskakis \cite{AaSa, AaSa1} and Chalmoukis \cite{Cn}. In that spirit, it is natural to introduce a new class of Carleson measures to characterize the boundedness and compactness of the generalized integral operators.

Let $\mathbb{D}$ be the open unit disc in the complex plane $\mathbb{C}$ and $H(\mathbb{D})$ be the space of all analytic functions on $\mathbb{D}$. For $0<p<\infty$, the classical Bergman space $A^p$ is defined as
\begin{equation*}
A^p=\left\{f\in H(\mathbb{D}):\|f\|_{p}^p=\int_{\mathbb{D}}|f(z)|^pdA(z)<\infty\right\},
\end{equation*}
where $dA=\frac{dxdy}{\pi}$ is the normalized area measure on $\mathbb{D}$. It is known that $A^p$ is a Banach space for $1\leq p<\infty$. In particular, $A^2$ is a Hilbert space with the following inner product
$$\langle f,g\rangle_2=\int_{\mathbb{D}}f(z)\overline{g(z)}dA(z).$$
For more information about Bergman spaces, one can refer to \cite{ZrZk,Zk}.

Recall that a positive Borel measure $\mu$ on $\mathbb{D}$ is called a $(k,p,q)$-Carleson measure for Bergman spaces if there exists a constant $C>0$ such that
$$\int_{\mathbb{D}}|f^{(k)}(z)|^qd\mu(z)\leq C\|f\|_p^q,$$
for all $f\in A^p$. Similarly, $\mu$ is called a vanishing $(k,p,q)$-Carleson measure if 
$$\lim_{j\to\infty}\int_{\mathbb{D}}|f_j^{(k)}(z)|^qd\mu(z)=0,$$
whenever $\{f_j\}$ is a bounded sequence in $A^p$ that converges to 0 uniformly on compact subsets of $\mathbb{D}$.

Fix a positive integer $n$ and let $\mathbf{u}=(u_0,u_1,\cdots,u_n)$ with $u_k\in H(\mathbb{D}) $ for $0\leq k\leq n$. We say that $\mu$ is a $(\mathbf{u},p,q)$-Sobolev Carleson measure for Bergman spaces if there exists a constant $C>0$ such that
\begin{equation*}
\int_{\mathbb{D}}|\langle\mathbf{u}(z),(\mathcal{D}_nf)(z)\rangle|^qd\mu(z)\leq C\|f\|_p^q,
\end{equation*}
for all $f\in A^p$, where $\mathcal{D}_nf=(f,f',\cdots,f^{(n-1)})$ and 
$$\langle \mathbf{u}, \mathcal{D}_nf\rangle=u_0f+u_1f'+\cdots+u_{n-1}f^{(n-1)}.$$
 Similarly, $\mu$ is called a vanishing $(\mathbf{u},p,q)$-Sobolev Carleson measure if 
\begin{equation*}
\lim_{j\to\infty}\int_{\mathbb{D}}\left|\langle\mathbf{u}(z),(\mathcal{D}_nf_j)(z)\rangle\right|^qd\mu(z)=0,
\end{equation*}
whenever $\{f_j\}$ is a bouned sequence in $A^p$ that converges to $0$ uniformly on compact subsets of $\mathbb{D}$.

The objective of this paper is first to characterize (vanishing) $(\mathbf{u},p,q)$-Sobolev Carleson measures for Bergman spaces. Then, by using these characterizations, we study several problems in function-theoretic operator theory such as the boundedness and compactness of generalized Volterra-type operators between different Bergman spaces and the sum of generalized weighted composition operators with different orders. Moreover, we give a sufficient condition for the solution to certain linear differential equations lying in Bergman spaces.

The first two main results are the following.

\begin{theorem}\label{Theorem1.1}
If $0<p\leq q<\infty$, then 
\begin{itemize}
\item[(a)] $\mu$ is a $(\mathbf{u},p,q)$-Sobolev Carleson measure if and only if each $|u_k|^qd\mu$, $0\leq k\leq n$, is a $(k,p,q)$-Carleson measure, and
\item[(b)] $\mu$ is a vanishing $(\mathbf{u},p,q)$-Sobolev Carleson measure if and only if each $|u_k|^qd\mu$, $0\leq k\leq n$, is a vanishing $(k,p,q)$-Carleson measure.
\end{itemize}
\end{theorem}
 
It is clear from the triangle inequality that if, for each $0\leq k\leq n$, $|u_k|^qd\mu$ is a $(k,p,q)$-Carleson measure, then $\mu$ is a $(\mathbf{u},p,q)$-Sobolev Carleson measure. The surprising part here is the converse. We mention that the corresponding result for Hardy spaces was partly obtained in \cite{Cn}.

\begin{theorem}\label{Theorem1.2}
If $0<q<p<\infty$, then the following statements are equivalent.
\begin{itemize}
\item[(a)] $\mu$ is a $(\mathbf{u},p,q)$-Sobolev Carleson measure;
\item[(b)] $\mu$ is a vanishing $(\mathbf{u},p,q)$-Sobolev Carleson measure;
\item[(c)] Each $|u_k|^qd\mu$, $0\leq k\leq n$, is a $(k,p,q)$-Carleson measure;
\item[(d)] Each $|u_k|^qd\mu$, $0\leq k\leq n$, is a vanishing $(k,p,q)$-Carleson measure.
\end{itemize}
\end{theorem}

The equivalence of $(c)$ and $(d)$ above is well known. See \cite{Ld2} for example.

\subsection{Generalized integral operator}
For $g\in H(\mathbb{D})$, the Volterra-type operator $I_g$ and its companion operator $J_g$ on $H(\mathbb{D})$ are defined by 
\begin{equation*}
(I_gf)(z)=\int_{0}^zf(w)g'(w)dw\quad {\rm and}\quad (J_gf)(z)=\int_{0}^zf'(w)g(w)dw. 
\end{equation*}
These operators have been studied extensively on various spaces of analytic functions in one and several variables. We refer the readers to \cite{AaSa, AaSa1, Co, GpGdPj, Xj} for more details.

Further, let $a=(a_0,a_1,\cdots,a_{n-1})\in\mathbb{C}^n$. The generalized integral operator $I_{g,a}$ is defined as follows:
$$(I_{g,a}f)(z)=I^n(a_0fg^{(n)}+a_1f'g^{n-1}+\cdots+a_{n-1}f^{(n-1)}g')(z),$$
where $I$ denotes the integration operator 
$$(If)(z)=\int_{0}^zf(w)dw,$$
and $I^n$ is the $n$th iteration of $I$. The operator $I_{g,a}$ was first introduced on Hardy spaces by Chalmoukis in \cite{Cn}. Chalmoukis completely characterized the boundedness and compactness of $I_{g,a}:H^p\to H^q$ when $0< p\leq q<\infty$, where $H^p$ is the classical Hardy space on $\mathbb{D}$. When $0<q<p<\infty$, he conjectured that $g\in H^{\frac{pq}{p-q}}$ if $I_{g,a}: H^p\to H^q$ is bounded. In this paper, we extend Chalmoukis' result to Bergman spaces and show that the Bergman space version of Chalmoukis' conjecture is true.

Instead of a single function, we now consider a vector-valued function $\mathbf{g}$, namely, $\mathbf{g}=(g_0,g_1,\cdots,g_{n-1})$, where $g_i\in H(\mathbb{D})$ for $i=0,1,\cdots,n-1$. We define a more generalized integral operator $I_{\mathbf{g}}^{(n)}$ on $H(\mathbb{D})$ as follows:
$$I_{\mathbf{g}}^{(n)}f=I^n(fg_0+f'g_1+\cdots+f^{(n-1)}g_{n-1}).$$

It is clear that $I_{\mathbf{g}}^{(n)}$ coincides with Chalmoukis' integral operator $I_{g,a}$ when $g_j=a_jg^{(n-j)}$ $(j=0,1,\cdots,n-1)$ for some $g\in H(\mathbb{D})$. Using Sobolev Carleson measure, we will completely characterize the boundedness and compactness of $I_{\mathbf{g}}^{(n)}$ and then $I_{g,a}$ from $A^p$ to $A^q$ as we are going to see in the main results and their corollaries below.

For $m\in \mathbb{N}$ and $\alpha>0$, we set 
\begin{equation*}
\mathcal{B}^{m,\alpha}=\left\{f\in H(\mathbb{D}):\|f\|_{m,\alpha}=\sup_{z\in\mathbb{D}}|f^{(m)}(z)|(1-|z|^2)^{\alpha}<\infty\right\}
\end{equation*}
and 
\begin{equation*}
\mathcal{B}_{0}^{m,\alpha}=\left\{f\in H(\mathbb{D}):\lim_{|z|\to 1^{-}}|f^{(m)}(z)|(1-|z|^2)^{\alpha}=0\right\}.
\end{equation*}
It is known that $\mathcal{B}^{0,\alpha}=\mathcal{B}^{m,m+\alpha}$. For more information about these spaces, we refer to \cite{LbOYc, Zk2}.

\begin{theorem}\label{Theorem1.3}
Let $n$ be a positive integer and $\mathbf{g}=(g_0,g_1,\cdots,g_{n-1})$ with $g_k\in H(\mathbb{C})$ for $0\leq k\leq n-1$. If $0<p\leq q<\infty$, then 
\begin{itemize}
\item[(a)] $I_{\mathbf{g}}^{(n)}:A^p\to A^q$ is bounded if and only if $g_k\in \mathcal{B}^{0,n-k-\frac{2}{p}+\frac{2}{q}}$ for $n-k\geq\frac{2}{p}-\frac{2}{q}$ and $g_k=0$ for $n-k<\frac{2}{p}-\frac{2}{q}$.
\item[(b)] $I_{\mathbf{g}}^{(n)}:A^p\to A^q$ is compact if and only if $g_k\in \mathcal{B}_{0}^{0,n-k-\frac{2}{p}+\frac{2}{q}}$ for $n-k>\frac{2}{p}-\frac{2}{q}$ and $g_k=0$ for $n-k\leq\frac{2}{p}-\frac{2}{q}$.
\end{itemize}
\end{theorem}

\begin{theorem}\label{Theorem1.4}
Let $n$ be a positive integer and $\mathbf{g}=(g_0,g_1,\cdots,g_{n-1})$ with $g_k\in H(\mathbb{C})$ for $0\leq k\leq n-1$. If $0<q<p<\infty$, then the following conditions are equivalent:
\begin{itemize}
\item[(a)] $I_{\mathbf{g}}^{(n)}:A^p\to A^q$ is bounded;
\item[(b)] $I_{\mathbf{g}}^{(n)}:A^p\to A^q$ is compact;
\item[(c)] For each $0\leq k\leq n-1$,
\begin{equation*}
\int_{\mathbb{D}}|g_k(z)(1-|z|^2)^{n-k}|^{\frac{pq}{p-q}}dA(z)<\infty.
\end{equation*}
\end{itemize}
\end{theorem}

As direct corollaries, we provide the following characterizations of $I_{g,a}:A^p\to A^q$.

\begin{corollary}\label{Coro1.5}
Suppose $g\in H(\mathbb{D})$ and $a=(a_0,a_1,\cdots,a_{n-1})\in \mathbb{C}^n$. If $0<p\leq q<\infty$ and $\frac{2}{p}-\frac{2}{q}<1$, then
\begin{itemize}
\item[(a)] $I_{g,a}:A^p\to A^q$ is bounded if and only if $g\in\mathcal{B}^{1,1-\frac{2}{p}+\frac{2}{q}}$.
\item[(b)] $I_{g,a}:A^p\to A^q$ is compact if and only if $g\in\mathcal{B}_0^{1,1-\frac{2}{p}+\frac{2}{q}}$.
\end{itemize}
\end{corollary}

\begin{corollary}\label{Coro1.6}
Suppose $g\in H(\mathbb{D})$ and $a=(a_0,a_1,\cdots,a_{n-1})\in\mathbb{C}^n$. If $0<q<p<\infty$, then the following conditions are equivalent:
\begin{itemize}
\item[(a)] $I_{g,a}:A^p\to A^q$ is bounded;
\item[(b)] $I_{g,a}:A^p\to A^q$ is compact;
\item[(c)] $g\in A^{\frac{pq}{p-q}}$.
\end{itemize}
\end{corollary}

\subsection{Linear differential equation}

In recent years, the research of the non-homogeneous linear complex differential equation 
\begin{equation}\label{equa1.1}
f^{(n)}+g_{n-1}(z)f^{(n-1)}+\cdots+g_0(z)f=F(z)
\end{equation}
in analytic function spaces has been widely concerned, where $F$ and $g_i$, $i=0,1,\cdots,n-1$ are analytic functions in some domains of $\mathbb{C}$. In \cite{Pc}, Pommerenke found a condition for the solution of the equation $f^{''}(z)+g(z)f(z)=0$. For the case of \eqref{equa1.1} with solutions in $H^p$, one can refer to \cite{Rj}. For the case of \eqref{equa1.1} with solutions in other function spaces, one can refer to \cite{GjHjRj, LhWl, PjRj, SyLjHg}. Inspired by \cite{Cn}, we get a sufficient condition for the solution of (1.1) belonging to the Bergman space $A^p$.

\begin{theorem}\label{Theorem1.7}
Suppose $F\in H(\mathbb{D})$ satisfies $I^nF\in A^p$ and $g_k\in \mathcal{B}^{0,n-k}$ for all $0\leq k\leq n-1$. Then there exists an $A>0$, depending only on $p$, such that every solution of the linear differential equation (1.1)
belongs to $A^p$ whenever $\|g_k\|_{0,n-k}<A$ for all $0\leq k\leq n-1$. Furthermore, if $g_k\in \mathcal{B}_{0}^{0,n-k}$ for all $0\leq k\leq n-1$, the same result holds without any restriction on the norm of $g_k$.
\end{theorem}

\subsection{Sums of generalized weighted composition operator}

Let $\varphi$ be an analytic self-map of $\mathbb{D}$ and $u\in H(\mathbb{D})$, the weighted composition operator $W_{u,\varphi}$ on $H(\mathbb{D})$ is defined by 
$$W_{u,\varphi}f=u\cdot f\circ\varphi.$$
It is known that the weighted composition operator is closely related to the linear isometric between Hardy and Bergman spaces. See for example \cite{Ff,Kc}.
When $\varphi(z)=z$, it reduces to the multiplication operator $M_u$. And when $u=1$, it reduces to the composition operator $C_{\varphi}$. The relationship between the operator-theoretic properties of $C_{\varphi}$ and the function-theoretic properties of $\varphi$ has been studied extensively during the past several decades. We refer the readers to the monographs \cite{CcMb,Sj} for more details. And one can refer to \cite{Ah, CbCkKhYj,CzZr,CzZr1} for the study of weighted composition operators on Bergman spaces.

Let $D$ be the differentiation operator on $H(\mathbb{D})$, i.e. $Df=f'$. $D^k$ is the $k$th iterate of $D$. The product of the operators $W_{u,\varphi}$ and $D^k$ is called the generalized weighted composition operator or weighted composition-differentiation operator of order $k$, denoted by $W_{u,\varphi}^{(k)}$, i.e.
\begin{equation*}
W_{u,\varphi}^{(k)}f=W_{u,\varphi}D^kf=u\cdot f^{(k)}\circ\varphi.
\end{equation*}
It is clear that $W_{u,\varphi}$ is a special case when $k=0$. When $k\geq 1$, the boundedness, compactness and Hilbert-Schmidtness of $W_{u,\varphi}^{(k)}$ between different weighted Bergman spaces were characterized in \cite{Zx,Zx1}. In addition, the differentiation-composition operator of order $k$ is denoted by $D_{\varphi}^{(k)}$, i.e. 
$$D_{\varphi}^{(k)}=D^kC_{\varphi}f=(f\circ\varphi)^{(k)}.$$
It is difficult to investigate the properties of $D_{\varphi}^{(k)}$ when $k\geq 2$.

In order to treat the product of the operators $M_u$, $C_{\varphi}$ and $D$ in a unified manner. Stevi\'c \cite{SsSaBa} et al. introduced the so-called Stevi\'s-Sharma operator $T_{u,v,\varphi}=W_{u,\varphi}+W_{v,\varphi}^{(1)}$ on weighted Bergman spaces. See also \cite{WsWmGx,YyLy,ZfLy} for more investigations about Stevi\'c-Sharma operators on several specific analytic function spaces.

Let $\mathbf{u}=(u_0,u_1,\cdots,u_n)$ with $u_k\in H(\mathbb{D})$ for $0\leq k\leq n$. Denote the sum operator $L_{\mathbf{u},\varphi}^{(n)}$ by 
$$L_{\mathbf{u},\varphi}^{(n)}=\sum_{k=0}^nW_{u_k,\varphi}^{(k)},$$
here $W_{u_0,\varphi}^{(0)}=W_{u_0,\varphi}$. It is clear that $T_{u,v,\varphi}$ is a special case of $L_{\mathbf{u},\varphi}^{(n)}$ when $n=1$, $u_0=u$ and $u_1=v$.

Just like the Sobolev Carleson measures, in this paper, we could also prove the rigidity of the operator $L_{\mathbf{u},\varphi}^{(n)}$. Specifically, the boundedness, compactness and Hilbert-Schmidtness of $L_{\mathbf{u},\varphi}^{(n)}$ are all equal to those of $W_{u_k,\varphi}^{(k)}$ for every $0\leq k\leq n$.
 
\begin{theorem}\label{Theorem1.8}
Let $0<p,q<\infty$, then the sum operator $L_{\mathbf{u},\varphi}^{(n)}:A^p\to A^q$ is bounded (or compact respectively) if and only if  $W_{u_k,\varphi}^{(k)}:A^p\to A^q$ is bounded (or compact respectively) for each $k=0,1,\cdots,n$. 
\end{theorem}

\begin{theorem}\label{Theorem1.9}
The sum operator $L_{\mathbf{u},\varphi}^{(n)}:A^2\to A^2$ is Hilbert-Schmidt if and only if each $W_{u_k,\varphi}^{(k)}:A^2\to A^2$,$k=0,1,\cdots,n$, is Hilbert-Schmidt.
\end{theorem}

The paper is organized as follows. In Section 2 we briefly give the preliminaries and background information, and some interesting key lemmas. In Sections 3 to 6, we prove our main results respectively as they are given above. 

Throughout the paper, we write $A\lesssim B$ (or equivalently $B\gtrsim A$) if there exists an absolute constant $C>0$ such that $A\leq B$. As usual, $A\simeq B$ means $A\lesssim B$ and $B\lesssim A$. We will be more specific if the dependence of such constants on certain parameters becomes critical.

\section{Preliminaries}

In this section, we present some preliminary facts and auxiliary lemmas which will be needed in our paper.

The following result is well known as the Little-Paley formula for Bergman spaces. See \cite{Zk,Zk1} for more details.

\begin{lemma}\label{lemma2.1}
Let $n$ be a positive integer, $0<p<\infty$ and $f\in H(\mathbb{D})$. Then $f\in A^p$ if and only if the function $(1-|z|^2)^nf^{(n)}(z)$ is in $L^p(\mathbb{D},dA)$. Moreover, the norm of $f$ in $A^p$ is comparable to 
\begin{equation*}
|f(0)|+|f'(0)|+\cdots+|f^{(n-1)}(0)|+\|(1-|z|^2)^nf^{(n)}\|_{L^p}.
\end{equation*}
\end{lemma}

Based on Lemma 2.1, one could easily obtain the following estimate for derivatives of functions in Bergman spaces. See \cite{Zx} for example.

\begin{lemma}\label{lemma2.2}
Let $p>0$ and $i$ be a non-negative integer. Then
\begin{equation*}
|f^{(i)}(z)|\lesssim \frac{\|f\|_p}{(1-|z|^2)^{i+\frac{2}{p}}},
\end{equation*}
for all $f\in A^p$ and $z\in\mathbb{D}$.
\end{lemma}

This lemma tells us that the linear point evaluation of $i$th order $f\mapsto f^{(i)}$ is bounded on $A^p$ for every $z\in\mathbb{D}$. In particular, by Riesz's representation theorem in Hilbert space theory, there exists a unique function $\mathcal{K}_z^{[i]}$ in $A^2$ such that 
$$f^{(i)}(z)=\langle f,\mathcal{K}_z^{[i]}\rangle_2,$$
for all $f\in A^2$. It is known that $\mathcal{K}_z(w)=\mathcal{K}_z^{[0]}(w)=\frac{1}{(1-\overline{z}w)^2}$ and $\mathcal{K}_z^{[i]}(w)=\frac{\partial^{i}\mathcal{K}_z}{\partial\overline{z}^i}(w)$ for $i\geq 1$. $\mathcal{K}_z^{[i]}$ is called the reproducing kernel function in $A^2$ at $z$ of order $i$.

For any $w\in\mathbb{D}$ and $i\in\mathbb{N}$, let 
$$K_w^{[i]}(z)=\frac{z^i}{(1-\overline{w}z)^{\gamma+i}},$$
where $\gamma>1+\frac{2}{p}$ is sufficiently large. Then we can get the estimates for the norm of $K_w^{[i]}$.

\begin{lemma}\label{lemma2.3}
Let $i\in\mathbb{N}$, then 
\begin{equation*}
\|K_w^{[i]}\|_p\simeq \frac{1}{(1-|w|^2)^{\gamma+i-\frac{2}{p}}},
\end{equation*}
for all $w\in\mathbb{D}$.
\end{lemma}

\begin{proof}
On the one hand, by \cite[Theorem 1.19]{Zk}, we have
\begin{equation}\label{equa2.1}
\begin{split}
\|K_w^{[i]}\|_p^p=\int_{\mathbb{D}}\left|\frac{z^i}{(1-\overline{w}z)^{\gamma+i}}\right|^pdA(z)\lesssim \frac{1}{(1-|w|^2)^{(\gamma+i)p-2}}.
\end{split}
\end{equation}
On the other hand, by lemma \ref{lemma2.2}, we have
\begin{equation}\label{equa2.2}
\|K_w^{[i]}\|_p\geq C(1-|w|^2)^{i+\frac{2}{p}}|(K_w^{[i]})^{(i)}(w)|\geq \frac{1}{(1-|w|^2)^{\gamma+i-\frac{2}{p}}}.
\end{equation}
Thus, the desired inequality follows from \eqref{equa2.1} and \eqref{equa2.2} immediately.
\end{proof}

By Lemma 2.3, we obtain that $\|\mathcal{K}_z^{[i]}\|_2\simeq \frac{1}{(1-|z|^2)^{1+i}},$  for all $z\in\mathbb{D}$ and $i\in\mathbb{N}$.

For any $a\in\mathbb{D}$, let $\varphi_a(z)=\frac{a-z}{1-\overline{a}z}$ be the involution automorphism that exchanges $0$ and $a$. The pseudo-hyperbolic metric on $\mathbb{D}$ is defined by
\begin{equation*}
\rho(z,w)=|\varphi_z(w)|,\quad z,w\in\mathbb{D}.
\end{equation*}
And the Bergman metric is defined by the formula
\begin{equation*}
\beta(z,w)=\frac{1}{2}\log\frac{1+\rho(z,w)}{1-\rho(z,w)},\quad z,w\in\mathbb{D}.
\end{equation*} 
For any $z\in\mathbb{D}$ and $r>0$, write $D(z,r)=\{w\in\mathbb{D}:\beta(z,w)<r\}$ for the Bergman metric disk centered at $z$ with radius $r$. It is well known that
\begin{equation}\label{equa2.3}
(1-|z|^2)^2\simeq (1-|w|^2)^2\simeq |1-\overline{w}z|^2\simeq |D(z,r)|,
\end{equation}
for all $z\in \mathbb{D}$ and $w\in D(z,r)$. Here $r$ is any fixed radius and $|D(z,r)|$ denotes the area of $D(z,r)$, See \cite{Zk} for example.

The geometric characterization for $(k,p,q)$-Carleson measures has been known for some time (see \cite{Ld2} and \cite{ZrZk}) and we state it as follows.

\begin{lemma}\label{lemma2.4}
Let $0<p,q<\infty$ and $\mu$ be a positive Borel measure on $\mathbb{D}$.
\begin{itemize}
\item[(i)] If $0<p\leq q<\infty$, then $\mu$ is a $(k,p,q)$-Carleson measure if and only if 
\begin{equation*}
\sup_{a\in\mathbb{D}}\frac{\mu(D(a,r))^{\frac{1}{q}}}{(1-|a|^2)^{k+\frac{2}{p}}}<\infty,
\end{equation*}
for some (or any) $r>0$, and $\mu$ is a vanishing $(k,p,q)$-Carleson measure if and only if 
\begin{equation*}
\lim_{|a|\to 1}\frac{\mu(D(a,r))^{\frac{1}{q}}}{(1-|a|^2)^{k+\frac{2}{p}}}=0,
\end{equation*}
for some (or any) $r>0$.
\item[(ii)] If $0<q<p<\infty$, then $\mu$ is a $(k,p,q)$-Carleson measure if and only if $\mu$ is a vanishing $(k,p,q)$-Carleson measure if and only if 
\begin{equation*}
\frac{\mu(D(z,r))}{(1-|z|^2)^{2+kq}}\in L^{\frac{p}{p-q}}(\mathbb{D},dA),
\end{equation*}
for some (or any) $r>0$.
\end{itemize}
\end{lemma}

A sequence $\{a_j\}$ in $\mathbb{D}$ is called a $r$-lattice in the Bergman metric if the following conditions are satisfied:
\begin{itemize}
\item[(i)] $\cup_{j=1}^{\infty}D(a_j,r)=\mathbb{D}$,
\item[(ii)] $\{D(a_j,r/2)\}_{j=1}^{\infty}$ are pairwise disjoint.
\end{itemize}
For any $r>0$, the existence of $r$-lattice in the Bergman metric can be ensured by \cite[Lemma 4.8]{Zk1}. And with hypotheses $(\rm i)$ and $(\rm ii)$, it is easy to check that
\begin{itemize}
\item[(iii)] 
there exists a positive integer $N$ such that every point in $\mathbb{D}$ belongs to at most $N$ of the sets $D(a_j,4r)$.
\end{itemize}

The following lemma is given by Amar in \cite{Ae}. Interested readers can also refer to \cite{Ld2} for a short introduction to Amar's work.

\begin{lemma}\label{lemma2.5}
Suppose $0<p<\infty$ and $\{a_j\}$ is a $r$-lattice in the Bergman metric of $\mathbb{D}$. Define an operator $S$ on complex sequences by 
\begin{equation*}
S(\{\lambda_j\})(z)=\sum_{j}\lambda_j\frac{(1-|a_j|^2)^{\gamma-\frac{2}{p}}}{(1-\overline{a_j}z)^{\gamma}},
\end{equation*}
where $\gamma>1+\frac{2}{p}$. Then $S$ is bounded from $l^p$ to $L^p(\mathbb{D},dA)$.
\end{lemma}

We end this section with the following proposition, which illustrates the key idea throughout this paper. We believe that this result should be of some independent value and interest.

\begin{proposition}\label{prop2.6}
Let $\alpha_0,\alpha_1,\cdots,\alpha_n\in\mathbb{C}\backslash\{0\}$, then 
\begin{equation*}
\left\|\sum_{i=0}^n\alpha_iK_w^{[i]}\right\|_p\gtrsim\sum_{i=0}^n|\alpha_i|\|K_w^{[i]}\|_p
\end{equation*}
for all $w\in\mathbb{D}$.
\end{proposition}

Before the proof of this proposition, we need an ingenious lemma. Given integers $0\leq n\leq m$ and $\beta\in (0,+\infty)$, set
$$\Gamma_{\beta}^{n,m}:=\prod_{j=n}^{m}(\beta+i).$$

\begin{lemma}\label{lemma2.7}
For any $w\in\mathbb{D}$, let 
$$F_w(z)=\frac{f(z)}{(1-\overline{w}z)^{\beta+n}},$$
where $\beta>0$, $n\in\mathbb{N}$ and $f\in H(\mathbb{D})$, then
\begin{equation}\label{equa3.01}
\begin{split}
&(1-|w|^2)^{n+1}\left(F_w^{(n+1)}(w)-\frac{f^{(n+1)(w)}}{(1-|w|^2)^{\beta+n}}\right)\\
&\quad+\sum_{i=0}^{n}(-1)^{i+1}\binom{n+1}{i+1}\Gamma_{\beta}^{n-i,n}\overline{w}^{i+1}F_{w}^{(n-1)}(w)(1-|w|^2)^{n-i}=0.
\end{split}
\end{equation}
\end{lemma}

\begin{proof}
By a direct calculation, we have
\begin{equation*}
\begin{split}
F_w^{(n+1)}(z)&=\frac{f^{(n+1)(z)}}{(1-\overline{w}z)^{\beta+n}}+\sum_{j=0}^{n}\binom{n+1}{j}f^{(j)}(z)\left(\frac{1}{(1-\overline{w}z)^{\beta+n}}\right)^{(n-j-1)}\\
&=\frac{f^{(n+1)(z)}}{(1-\overline{w}z)^{\beta+n}}+\sum_{j=0}^{n}\binom{n+1}{j}f^{(j)}(z)\frac{\Gamma_{\beta}^{n,2n-j}\overline{w}^{n-j-1}}{(1-\overline{w}z)^{\beta+2n-j+1}}.
\end{split}
\end{equation*}
Then we write 
\begin{equation*}
(1-|w|^2)^{n+1}\left(F_{w}^{(n+1)}(w)-\frac{f^{(n+1)}(w)}{(1-|w|^2)^{\beta+n}}\right)=A_1(w)+B_1(w)+C_1(w),
\end{equation*}
where 
$$A_1(w)=f(w)\frac{\Gamma_{\beta}^{n,2n}\overline{w}^{n+1}}{(1-|w|^2)^{\beta+n}},\quad B_1(w)=\frac{(n+1)(\beta+n)f^{(n)}(w)\overline{w}}{(1-|w|^2)^{\beta}}$$
and 
$$C_1(w)=\sum_{j=1}^{n-1}\binom{n+1}{j}f^{(j)}(w)\frac{\Gamma_{\beta}^{n,2n-j}\overline{w}^{n-j+1}}{(1-|w|^2)^{\beta+n-j}}.$$
Note that
\begin{equation*}
\begin{split}
&\quad\sum_{i=0}^n(-1)^{i+1}\binom{n+1}{i+1}\Gamma_{\beta}^{n-i,n}\overline{w}^{i+1}F_w^{(n-i)}(w)(1-|w|^2)^{n-i}\\
&=\sum_{i=0}^{n-1}(-1)^{i+1}\binom{n+1}{i+1}\Gamma_{\beta}^{n-i,n}\overline{w}^{i+1}F_w^{(n-i)}(w)(1-|w|^2)^{n-i}\\
&\qquad+(-1)^{n+1}\Gamma_{\beta}^{0,n}\overline{w}^{n+1}\frac{f(w)}{(1-|w|^2)^{\beta+n}}
\end{split}
\end{equation*}
and 
\begin{equation*}
F_w^{(n-i)}(z)=\frac{f^{(n-i)}(z)}{(1-\overline{w}z)^{\beta+n}}+\sum_{j=0}^{n-1-i}\binom{n-i}{j}f^{(j)}(z)\frac{\Gamma_{\beta}^{n,2n-i-j-1}\overline{w}^{n-i-j}}{(1-\overline{w}z)^{\beta+2n-i-j}}.
\end{equation*}
It follows that 
\begin{equation*}
\begin{split}
&\quad\sum_{i=0}^{n-1}(-1)^{i+1}\binom{n+1}{i+1}\Gamma_{\beta}^{n-i,n}\overline{w}^{i+1}F_{w}^{(n-i)}(w)(1-|w|^2)^{n-i}\\
&=\sum_{i=0}^{n-1}\sum_{j=0}^{n-1-i}\frac{(-1)^{i+1}\binom{n+1}{i+1}\binom{n-i}{j}\Gamma_{\beta}^{n-i,n}\Gamma_{\beta}^{n,2n-i-j-1}\overline{w}^{n-j+1}f^{(j)}(w)}{(1-|w|^2)^{\beta+n-j}}\\
&\qquad+\sum_{i=0}^{n-1}(-1)^{i+1}\binom{n+1}{i+1}\Gamma_{\beta}^{n-i,n}\overline{w}^{i+1}\frac{f^{(n-i)}(w)}{(1-|w|^2)^{\beta+i}}\\
&=\sum_{j=0}^{n-1}\sum_{i=0}^{n-1-j}\frac{(-1)^{i+1}\binom{n+1}{i+1}\binom{n-i}{j}\Gamma_{\beta}^{n-i,n}\Gamma_{\beta}^{n,2n-i-j-1}\overline{w}^{n-j+1}f^{(j)}(w)}{(1-|w|^2)^{\beta+n-j}}\\
&\qquad+\sum_{i=0}^{n-1}(-1)^{i+1}\binom{n+1}{i+1}\Gamma_{\beta}^{n-i,n}\overline{w}^{i+1}\frac{f^{(n-i)}(w)}{(1-|w|^2)^{\beta+i}}.
\end{split}
\end{equation*}
Now we can write
\begin{equation*}
\begin{split}
&\quad\sum_{i=0}^{n}(-1)^{i+1}\binom{n+1}{i+1}\Gamma_{\beta}^{n-i,n}\overline{w}^{i+1}F_w^{(n-i)}(w)(1-|w|^2)^{n-i}\\
&=A_2(w)+B_2(w)+C_2(w),
\end{split}
\end{equation*}
where
\begin{equation*}
A_2(w)=\left((-1)^{n+1}\Gamma_{\beta}^{0,n}+\sum_{i=0}^{n-1}(-1)^{i+1}\binom{n+1}{i+1}\Gamma_{\beta}^{n-i,n}\Gamma_{\beta}^{n,2n-i-1}\right)\frac{f(w)\overline{w}^{n+1}}{(1-|w|^2)^{\beta+n}},
\end{equation*}
$$B_2(w)=-(n+1)(\beta+n)\frac{f^{(n)}(w)\overline{w}}{(1-|w|^2)^{\beta}}$$
and
\begin{equation*}
\begin{split}
C_2(w)&=\sum_{j=1}^{n-1}\sum_{i=0}^{n-1-j}(-1)^{i+1}\binom{n+1}{i+1}\binom{n-i}{j}\Gamma_{\beta}^{n-i,n}\Gamma_{\beta}^{n,2n-i-j-1}\frac{f^{(j)}(w)\overline{w}^{n+1-j}}{(1-|w|^2)^{\beta+n-j}}\\
&\quad+\sum_{i=1}^{n-1}(-1)^{i+1}\binom{n+1}{i+1}\Gamma_{\beta}^{n-i,n}\frac{f^{(n-i)}(w)\overline{w}^{i+1}}{(1-|w|^2)^{\beta+i}}\\
&=\sum_{j=1}^{n-1}\sum_{i=0}^{n-1-j}(-1)^{i+1}\binom{n+1}{i+1}\binom{n-i}{j}\Gamma_{\beta}^{n-i,n}\Gamma_{\beta}^{n,2n-i-j-1}\frac{f^{(j)}(w)\overline{w}^{n+1-j}}{(1-|w|^2)^{\beta+n-j}}\\
&\quad+\sum_{j=1}^{n-1}(-1)^{n-j+1}\binom{n+1}{j}\Gamma_{\beta}^{j,n}\frac{f^{(j)}(w)\overline{w}^{n+1-j}}{(1-|w|^2)^{\beta+n-j}}.
\end{split}
\end{equation*}

Note that $B_1(w)+B_2(w)=0$, it is sufficient to prove that
\begin{equation*}
A_1(w)+A_2(w)=0\quad {\rm and} \quad C_1(w)+C_2(w)=0.
\end{equation*}
It is then sufficient to prove
\begin{equation}\label{equa2.5}
\Gamma_{\beta}^{n,2n}+\sum_{i=0}^{n-1}(-1)^{i+1}\binom{n+1}{i+1}\Gamma_{\beta}^{n-i,n}\Gamma_{\beta}^{n,2n-i-1}+(-1)^{n+1}\Gamma_{\beta}^{0,n}=0
\end{equation}
and 
\begin{equation}\label{equa2.6}
\begin{split}
&\binom{n+1}{j}\Gamma_{\beta}^{n,2n-j}+\sum_{i=0}^{n-1-j}(-1)^{i+1}\binom{n+1}{i+1}\binom{n-i}{j}\Gamma_{\beta}^{n-i,n}\Gamma_{\beta}^{n,2n-i-j-1}\\
&\quad+(-1)^{n-j+1}\binom{n+1}{j}\Gamma_{\beta}^{j,n}=0,
\end{split}
\end{equation}
for $j=1,2,\cdots,n-1$. Note that $\binom{n+1}{i+1}\binom{n-i}{j}=\binom{n+1}{j}\binom{n+1-j}{i+1}$ and \eqref{equa2.5} is the case for $j=0$ in \eqref{equa2.6}. Therefore, we only need to prove
\begin{equation*}
b_j=\Gamma_{\beta}^{n,2n-j}+\sum_{i=1}^{n-j}(-1)^i\binom{n+1-j}{i}\Gamma_{\beta}^{n-i+1,n}\Gamma_{\beta}^{n,2n-i-j}+(-1)^{n-j+1}\Gamma_{\beta}^{j,n}=0
\end{equation*}
for every $j=0,1,\cdots,n-1$. In fact, let $1=g(x)=h_1(x)h_2(x)$, where $h_1(x)=\frac{1}{(1-x)^{\beta+n}}$ and $h_2(x)=(1-x)^{\beta+n}$, then 
\begin{equation*}
\begin{split}
b_j&=h_1^{(n+1-j)}(0)h_2(0)+\sum_{i=1}^{n-j}\binom{n+1-j}{i}h_1^{(n+1-j-i)}(0)h_2^{(i)}(0)+h_1(0)h_2^{(n+1-j)}(0)\\
&=g^{(n+1-j)}(0)=0,
\end{split}
\end{equation*}
for $j=0,1,\cdots,n-1$. The proof is now complete.
\end{proof}

\begin{proof}[Proof of Proposition 2.6]
The case $n=0$ is trivial. 

For $n=1$, by Lemma \ref{lemma2.2}, we have
\begin{equation}\label{equa2.7}
\begin{split}
\|\alpha_0K_w^{[0]}+\alpha_1K_w^{[1]}\|_p&\gtrsim\sup_{z\in\mathbb{D}}(1-|z|^2)^{\frac{2}{p}}\left|\alpha_0K_w^{[0]}(z)+\alpha_1K_w^{[1]}(z)\right|\\
&\geq (1-|w|^2)^{\frac{2}{p}}\left|\alpha_0K_w^{[0]}(w)+\alpha_1K_w^{[1]}(w)\right|\\
&=\left|\alpha_0\frac{1}{(1-|w|^2)^{\gamma-\frac{2}{p}}}+\alpha_1\frac{w}{(1-|w|^2)^{\gamma+1-\frac{2}{p}}}\right|
\end{split}
\end{equation}
and
\begin{equation}\label{equa2.8}
\begin{split}
\|\alpha_0K_w^{[0]}+\alpha_1K_w^{[1]}\|_p&\gtrsim\sup_{z\in\mathbb{D}}(1-|z|^2)^{\frac{2}{p}+1}\left|\left(\alpha_0K_w^{[0]}+\alpha_1K_w^{[1]}\right)'(z)\right|\\
&\geq(1-|w|^2)^{\frac{2}{p}+1}\left|\left(\alpha_0K_w^{[0]}+\alpha_1K_w^{[1]}\right)'(w)\right|\\
&=\left|\alpha_0\frac{\gamma\overline{w}}{(1-|w|^2)^{\gamma-\frac{2}{p}}}+\alpha_1\frac{\gamma|w|^2+1}{(1-|w|^2)^{\gamma+1-\frac{2}{p}}}\right|.
\end{split}
\end{equation}
Let $|\gamma\overline{w}|(\rm 2.7)+(\rm 2.8)$, using triangle inequality, we obtain
\begin{equation*}
\|\alpha_1K_w^{[1]}\|_p\simeq \frac{|\alpha_1|}{(1-|w|^2)^{\gamma+2-\frac{2}{p}}}\lesssim\left\|\alpha_0K_w^{[0]}+\alpha_1K_w^{[1]}\right\|_p.
\end{equation*}
It follows immediately that
\begin{equation*}
\left\|\alpha_0K_w^{[0]}\right\|_p\lesssim \left\|\alpha_0K_w^{[0]}+\alpha_1K_w^{[1]}\right\|_p.
\end{equation*}

The process above can be regarded as the following Gaussian elimination by rows. Let
$$M_2(z)=A_2(z)B_{2,w}(z)C_{2,\alpha},$$
where 
\begin{align*}
A_2(z)=
\begin{bmatrix}
(1-|z|^2)^{\frac{2}{p}}&0\\
0&(1-|z|^2)^{\frac{2}{p}+1}
\end{bmatrix},
\quad
C_{2,\alpha}=
\begin{bmatrix}
\alpha_0&0\\
0&\alpha_1
\end{bmatrix}
\end{align*}
and
\begin{align*}
B_{2,w}(z)=
&\left[\begin{array}{ll}
K_w^{[0]}(z)&K_w^{[1]}(z)\\
(K_w^{[0]})'(z)&(K_w^{[1]})'(z)
\end{array}\right]\\
=&\left[\begin{array}{ll}
\frac{1}{(1-\overline{w}z)^{\gamma}}&\frac{1}{(1-\overline{w}z)^{\gamma}}\cdot\frac{z}{1-\overline{w}z}\\
\frac{\gamma\overline{w}}{(1-\overline{w}z)^{\gamma+1}}\quad&\frac{\gamma \overline{w}}{(1-\overline{w}z)^{\gamma+1}}\cdot\frac{z}{1-\overline{w}z}+\frac{1}{(1-\overline{w}z)^{\gamma}}\cdot\frac{1}{(1-\overline{w}z)^2}
\end{array}\right].
\end{align*}
Therefore,
\begin{align*}
M_2(w)=
&\begin{bmatrix}
\alpha_0\frac{1}{(1-|w|^2)^{\gamma-\frac{2}{p}}}&\alpha_1\frac{w}{(1-|w|^2)^{\gamma+1-\frac{2}{p}}}\\
\alpha_0\frac{\gamma\overline{w}}{(1-|w|^2)^{\gamma-\frac{2}{p}}}&\alpha_1\frac{\gamma|w|^2+1}{(1-|w|^2)^{\gamma+1-\frac{2}{p}}}
\end{bmatrix}\\
\xrightarrow{\mathbf{r}_2-\gamma\overline{w}\mathbf{r}_1}
&\begin{bmatrix}
\alpha_0\frac{1}{(1-|w|^2)^{\gamma-\frac{2}{p}}}&\alpha_1\frac{w}{(1-|w|^2)^{\gamma+1-\frac{2}{p}}}\\
0&\alpha_1\frac{1}{(1-|w|^2)^{\gamma+1-\frac{2}{p}}}
\end{bmatrix}.
\end{align*}
This means that
\begin{equation*}
|\alpha_1|\frac{1}{(1-|w|^2)^{\gamma+1-\frac{2}{p}}}\lesssim\left\|\alpha_0K_w^{[0]}+\alpha_1K_w^{[1]}\right\|_p.
\end{equation*}

For the general case, by induction, we only need to prove that
\begin{equation}\label{equa2.9}
\left\|\sum_{i=0}^n\alpha_iK_w^{[i]}\right\|_p\gtrsim|\alpha_n|\|K_w^{[n]}\|_p\simeq\frac{|\alpha_n|}{(1-|w|^2)^{\gamma+n-\frac{2}{p}}}.
\end{equation}
After we achieve \eqref{equa2.9}, we can see
\begin{equation*}
\left\|\sum_{i=0}^{n-1}\alpha_iK_w^{[i]}\right\|_p\leq \left\|\sum_{i=0}^{n}\alpha_iK_w^{[i]}\right\|_p+|\alpha_n|\|K_w^{[n]}\|_p\lesssim\left\|\sum_{i=0}^{n}\alpha_iK_w^{[i]}\right\|_p.
\end{equation*}
Then, the induction works to get
\begin{equation*}
|\alpha_j|\|K_w^{[j]}\|_p\lesssim\left\|\sum_{i=0}^{n-1}\alpha_iK_w^{[i]}\right\|_p\lesssim\left\|\sum_{i=0}^n\alpha_iK_w^{[i]}\right\|_p
\end{equation*}
for $j=0,1,\cdots,n-1$.

To prove \eqref{equa2.9}, we conduct similar Gaussian elimination by rows. Let
\begin{align*}
M_{n+1}(z)=A_{n+1}(z)B_{n+1,w}(z)C_{n+1,\alpha},
\end{align*}
where $A_{n+1}(z)=diag\left\{(1-|z|^2)^{\frac{2}{p}},(1-|z|^2)^{\frac{2}{p}+1},\cdots,(1-|z|^2)^{\frac{2}{p}+n}\right\}$, $C_{n+1,\alpha}=diag\{\alpha_0,\alpha_1,\cdots,\alpha_n\}$, and 
$$B_{n+1,w}(z)=\left((K_w^{[i]})^{(j)}(z)\right)_{0\leq i,j\leq n}.$$
We can write $K_w^{[i]}(z)=Q_{n-1}(z)P_{i,n-1-i}(z)$, where $Q_{n-1}(z)=\frac{1}{(1-\overline{w}z)^{\gamma+n-1}}$ and $P_{i,n-1-i}(z)=z^i(1-\overline{w}z)^{n-1-i}$ for $i=0,1,\cdots,n$.

Now taking $f(z)=P_{i,n-1-i}(z)$ in Lemma \ref{lemma2.7}, $z=w$ in the matrix $M_{n+1}(z)$ and letting 
$$\mathbf{r}_{n+1}'=\mathbf{r}_{n+1}+\sum_{i=1}^{n}(-1)^{i}\binom{n}{i}(\gamma+n-1)\cdots(\gamma+n-i)\overline{w}^{i}\mathbf{r}_i,$$
we get the matrix $M_{n+1}'$ as the following form 
\begin{align*}
M_{n+1}^{'}(w)=\left[
\begin{array}{lllll}
a_{1,1}(w)&a_{1,2}(w)&\cdots&a_{1,n}(w)&a_{1,n+1}(w)\\
a_{2,1}(w)&a_{2,2}(w)&\cdots&a_{2,n}(w)&a_{2,n+1}(w)\\
\vdots&\vdots&\ddots&\vdots&\vdots\\
a_{n,1}(w)&a_{n,2}(w)&\cdots&a_{n,n}(w)&a_{n,n+1}(w)\\
0&0&\cdots&0&\alpha_n\frac{n!}{(1-|w|^2)^{\gamma+n-\frac{2}{p}}}
\end{array}
\right].
\end{align*}
This means that
\begin{equation*}
\frac{|\alpha_n|}{(1-|w|^2)^{\gamma+n-\frac{2}{p}}}\lesssim \left\|\sum_{i=0}^{n}\alpha_iK_w^{[i]}\right\|_p.
\end{equation*}
Then,  the proof completes by induction.
\end{proof}

\section{Proof of Theorem 1.1 and Theorem 1.2}

In this section, we present the proof of Theorem 1.1 and Theorem 1.2 about Sobolev-Carleson measures.

\begin{proof}[\bf Proof of Theorem 1.1]
The ``if part" of $(a)$ follows easily from the triangle inequality and the definitions of $(k,p,q)$-Carleson measures and $(\mathbf{u},p,q)$-Sobolev Carleson measures. Similarly, the ``if part" of $(b)$ follows from the triangle inequality and the definitions of vanishing $(k,p,q)$-Carleson measures and vanishing $(\mathbf{u},p,q)$-Sobolev Carleson measures.

For the proof of the ``only if part" of $(a)$, note that the case $n=0$ is trivial. When $n=1$, let $k_w^{[i]}(z)=(1-|w|^2)^{\gamma+i-\frac{2}{p}}K_w^{[i]}(z)$, $i=0,1$. Applying Lemma \ref{lemma2.3}, we get
\begin{equation}\label{equa3.1}
\sup_{w\in\mathbb{D}}\int_{\mathbb{D}}\left|u_0(z)k_w^{[0]}(z)+u_1(z)(k_w^{[0]})'(z)\right|^qd\mu(z)<\infty
\end{equation}
and
\begin{equation}\label{equa3.2}
\sup_{w\in\mathbb{D}}\int_{\mathbb{D}}\left|u_0(z)k_w^{[1]}(z)+u_1(z)(k_w^{[1]})'(z)\right|^qd\mu(z)<\infty.
\end{equation}
Fix $r>0$, the estimate in \eqref{equa3.1} clear gives
\begin{equation*}
\sup_{w\in\mathbb{D}}\int_{D(w,r)}\left|\frac{u_0(z)}{(1-\overline{w}z)^{\gamma}}+\frac{\gamma\overline{w}u_1(z)}{(1-\overline{w}z)^{\gamma+1}}\right|^q(1-|w|^2)^{(\gamma-\frac{2}{p})q}d\mu(z)<\infty.
\end{equation*}
Recall from \eqref{equa2.3} that $|1-\overline{w}z|\simeq(1-|w|^2)$ for $z\in D(w,r)$. Thus
\begin{equation}\label{equa3.3}
\sup_{w\in\mathbb{D}}\int_{D(w,r)}\left|u_0(z)+\frac{\gamma\overline{w}u_1(z)}{1-\overline{w}z}\right|^q\frac{1}{(1-|w|^2)^{\frac{2q}{p}}}d\mu(z)<\infty.
\end{equation}
Similarly, it follows from \eqref{equa3.2} that
\begin{equation}\label{equa3.4}
\sup_{w\in\mathbb{D}}\int_{D(w,r)}\left|zu_0(z)+\frac{(1+\gamma\overline{w}z)u_1(z)}{1-\overline{w}z}\right|^q\frac{1}{(1-|w|^2)^{\frac{2q}{p}}}d\mu(z)<\infty.
\end{equation}
Adding $|z|\cdot$(3.3) and (3.4), we obtain
\begin{equation}\label{equa3.5}
\sup_{w\in\mathbb{D}}\int_{D(w,r)}\left|\frac{u_1(z)}{1-\overline{w}z}\right|^q\frac{1}{(1-|w|^2)^{\frac{2q}{p}}}d\mu(z)<\infty.
\end{equation}
Another application of \eqref{equa2.3} yields
\begin{equation*}
\sup_{w\in\mathbb{D}}\frac{\left(\int_{D(w,r)}|u_1(z)|^qd\mu(z)\right)^{\frac{1}{q}}}{(1-|w|^2)^{1+\frac{2}{p}}}<\infty.
\end{equation*}
This exactly shows that $|u_1|^qd\mu$ is a $(1,p,q)$-Carleson measure by Lemma \ref{lemma2.4}. From \eqref{equa3.3} and \eqref{equa3.5}, we also have 
\begin{equation*}
\sup_{w\in\mathbb{D}}\frac{\left(\int_{D(w,r)}|u_0(z)|^qd\mu(z)\right)^{\frac{1}{q}}}{(1-|w|^2)^{\frac{2}{p}}}<\infty,
\end{equation*}
which shows that $|u_0|^qd\mu$ is a $(0,p,q)$-Carleson measure.

Modifying the proof of Proposition \ref{prop2.6}, the general case $n\geq 2$ can also be proved by induction. 

And the proof for the ``only if part" of (b) is similar, we omit the routine details.
\end{proof}

We now proceed to the proof of Theorem 1.2. Before that, we recall some basic facts of the classical Khinchine's inequality, which is an important tool in complex and functional analysis. An introduction to the topic can be found in Appendix A of \cite{Dp}.

Let $\{r_{k}(t)\}$  denote the sequence of Rademacher functions defined by
$$r_0(t)=\begin{cases}1, &\mbox{if}\ 0\leq t-[t]<\frac{1}{2}\\-1, &\mbox{if}\ \frac{1}{2}\leq t-[t]<1,\end{cases}$$ 
where $[t]$ denotes the largest integer not greater than $t$ and $r_k(t)=r_0(2^kt)$ for $k=1.2,\cdots$.
If $0<p<\infty$, then Khinchine's inequality states that
$$\left(\sum_{k}\left|c_{k}\right|^{2}\right)^{p / 2} \simeq \int_{0}^{1}\left|\sum_{k} c_{k} r_{k}(t)
\right|^{p} dt,$$
for complex sequences $\left\{c_{k}\right\}$.

\begin{proof}[\bf Proof of Theorem 1.2]
The equivalence of (c) and (d) was given in Lemma \ref{lemma2.4}. It is obvious (b) implies (a). That (c) implies (a) and (d) implies (b) follow from the triangle inequality. Thus we only need to prove that (a) implies (c).

To this end, let $\{a_j\}$ be any $r$-lattice in Bergman metric and $\{\lambda_j\}\in l^p$, we set
\begin{equation*}
h_k(z)=\sum_{j}\lambda_j\frac{z^k(1-|a_j|^2)^{\gamma+k-\frac{2}{p}}}{(1-\overline{a}_jz)^{\gamma+k}}, \quad k=0,1,\cdots,n.
\end{equation*}
By Lemma \ref{lemma2.5}, we have
\begin{equation*}
\|h_k\|_p\lesssim \left(\sum_{j}|\lambda_j|^p\right)^{1/p},\quad k=0,1,\cdots, n.
\end{equation*}
For $n=1$, we obtain
\begin{equation}\label{equa3.6}
\int_{\mathbb{D}}|u_0(z)h_0(z)+u_1(z)h_0'(z)|^qd\mu(z)\lesssim \left(\sum_{j}|\lambda_j|^p\right)^{q/p}
\end{equation}
and 
\begin{equation}\label{equa3.7}
\int_{\mathbb{D}}|u_0(z)h_1(z)+u_1(z)h_1'(z)|^qd\mu(z)\lesssim \left(\sum_{j}|\lambda_j|^p\right)^{q/p}.
\end{equation}
In \eqref{equa3.6}, we replace $\lambda_j$ by $r_j(t)\lambda_j$, so that the right-hand side does not change. Then we integrate both sides with respect to $t$ from 0 to 1 to obtain
\begin{equation*}
\begin{split}
&\quad\int_{0}^1\int_{\mathbb{D}}\left|u_0(z)\sum_{j}r_j(t)\lambda_j\frac{(1-|a_j|^2)^{\gamma-\frac{2}{p}}}{(1-\overline{a_j}z)^{\gamma}}\right.\\
&\quad\quad\qquad\qquad\left.+u_1(z)\sum_{j}r_j(t)\lambda_j\frac{\gamma\overline{a_j}(1-|a_j|^2)^{\gamma-\frac{2}{p}}}{(1-\overline{a_j}z)^{\gamma+1}}\right|^qd\mu(z)dt\\
&\lesssim \left(\sum_{j}|\lambda_j|^p\right)^{q/p}.
\end{split}
\end{equation*}
By Khinchine's inequality and Fubini's Theorem,
\begin{equation*}
\begin{split}
&\int_{\mathbb{D}}\left(\sum_{j}|\lambda_j|^2\left|u_0(z)\frac{(1-|a_j|^2)^{\gamma-\frac{2}{p}}}{(1-\overline{a_j}z)^{\gamma}}+u_1(z)\frac{\gamma\overline{a_j}(1-|a_j|^2)^{\gamma-\frac{2}{p}}}{(1-\overline{a_j}z)^{\gamma+1}}\right|^2\right)^{\frac{q}{2}}d\mu(z)\\
\leq&\int_{\mathbb{D}}\int_{0}^1\left|u_0(z)\sum_{j}r_j(t)\lambda_j\frac{(1-|a_j|^2)^{\gamma-\frac{2}{p}}}{(1-\overline{a_j}z)^{\gamma}}\right.\\
&\qquad\qquad\quad\quad\qquad\qquad\left.+u_1(z)\sum_{j}r_j(t)\lambda_j\frac{\gamma\overline{a_j}(1-|a_j|^2)^{\gamma-\frac{2}{p}}}{(1-\overline{a_j}z)^{\gamma+1}}\right|^qdtd\mu(z)\\
=&\int_{0}^1\int_{\mathbb{D}}\left|u_0(z)\sum_{j}r_j(t)\lambda_j\frac{(1-|a_j|^2)^{\gamma-\frac{2}{p}}}{(1-\overline{a_j}z)^{\gamma}}\right.\\
&\qquad\qquad\quad\quad\qquad\qquad\left.+u_1(z)\sum_{j}r_j(t)\lambda_j\frac{\gamma\overline{a_j}(1-|a_j|^2)^{\gamma-\frac{2}{p}}}{(1-\overline{a_j}z)^{\gamma+1}}\right|^qd\mu(z)dt\\
&\lesssim \left(\sum_{j}|\lambda_j|^p\right)^{q/p}.
\end{split}
\end{equation*}
It follows from \eqref{equa2.3} that $\chi_{D(a_j,r)}(z)\lesssim (1-|a_j|^2)/|1-\overline{a_j}z|$. Thus
\begin{equation*}
\begin{split}
&\quad\int_{\mathbb{D}}\left(\sum_{j}|\lambda_j|^2\chi_{D(a_j,r)}(z)\left|u_0(z)+\frac{\gamma\overline{a_j}u_1(z)}{1-\overline{a_j}z}\right|^2\frac{1}{(1-|a_j|^2)^{\frac{4}{p}}}\right)^{\frac{q}{2}}d\mu(z)\\
&\lesssim \left(\sum_{j}|\lambda_j|^p\right)^{q/p}.
\end{split}
\end{equation*}
Recall that there is a positive integer $N$ such that each point $z\in\mathbb{D}$ belongs to at most $N$ of the disks $D(a_j,r)$. Applying Minkowski's inequality if $\frac{2}{q}\leq 1$ and H$\rm\ddot o$lder's inequality if $\frac{2}{q}>1$, we obtain
\begin{equation}\label{equa3.8}
\begin{split}
&\quad\sum_{j}|\lambda_j|^q\int_{D(a_j,r)}\left|u_0(z)+\frac{u_1(z)\gamma\overline{a}_j}{1-\overline{a}_jz}\right|^q\frac{1}{(1-|a_j|^2)^{\frac{2q}{p}}}d\mu(z)\\
&\leq \int_{\mathbb{D}}\sum_{j}|\lambda_j|^q\chi_{D(a_j,r)}(z)\left|u_0(z)+\frac{u_1(z)\gamma\overline{a}_j}{1-\overline{a}_jz}\right|^q\frac{1}{(1-|a_j|^2)^{\frac{2q}{p}}}d\mu(z)\\
&\leq C\int_{\mathbb{D}}\left(\sum_{j}|\lambda_j|^2\chi_{D(a_j,r)}(z)\left|u_0(z)+\frac{\gamma\overline{a}_ju_1(z)}{1-\overline{a}_jz}\right|^2\frac{1}{(1-|a_j|^2)^{\frac{4}{p}}}\right)^{\frac{q}{2}}d\mu(z)\\
&\lesssim \left(\sum_{j}|\lambda_j|^p\right)^{q/p},
\end{split}
\end{equation}
where $C=\max\{N^{1-\frac{q}{2}},1\}$. It follows immediately that
\begin{equation}\label{equa3.9}
\begin{split}
&\quad\sum_{j}|\lambda_j|^q\int_{D(a_j,r)}\left|u_0(z)z+\frac{u_1(z)\gamma\overline{a}_jz}{1-\overline{a}_jz}\right|^q\frac{1}{(1-|a_j|^2)^{\frac{2q}{p}}}d\mu(z)\\
&\lesssim \left(\sum_{j}|\lambda_j|^p\right)^{q/p}.
\end{split}
\end{equation}
Similarly, from \eqref{equa3.7}, we deduce that
\begin{equation}\label{equa3.10}
\begin{split}
&\quad\sum_{j}|\lambda_j|^q\int_{D(a_j,r)}\left|u_0(z)z+\frac{u_1(z)(1+\gamma\overline{a}_jz)}{1-\overline{a}_jz}\right|^q\frac{1}{(1-|a_j|^2)^{\frac{2q}{p}}}d\mu(z)\\
&\lesssim \left(\sum_{j}|\lambda_j|^p\right)^{q/p}.
\end{split}
\end{equation}
Adding \eqref{equa3.9} and \eqref{equa3.10}, we obtain
\begin{equation}\label{equa3.11}
\sum_{j}|\lambda_j|^q\int_{D(a_j,r)}\left|\frac{u_1(z)}{1-\overline{a}_jz}\right|^q\frac{1}{(1-|a_j|^2)^{\frac{2q}{p}}}d\mu(z)\lesssim \left(\sum_{j}|\lambda_j|^p\right)^{q/p}.
\end{equation}
Another application of \eqref{equa2.3} yields
\begin{equation*}
\sum_{j}|\lambda_j|^q\int_{D(a_j,r)}\left|\frac{u_1(z)}{(1-|a_j|^2)^{\frac{2}{p}+1}}\right|^qd\mu(z)\lesssim \left(\sum_{j}|\lambda_j|^p\right)^{q/p}.
\end{equation*}
And \eqref{equa3.11} together with \eqref{equa3.8} yields
\begin{equation*}
\sum_{j}|\lambda_j|^q\int_{D(a_j,r)}\left|\frac{u_0(z)}{(1-|a_j|^2)^{\frac{2}{p}}}\right|^qd\mu(z)\lesssim \left(\sum_{j}|\lambda_j|^p\right)^{q/p}.
\end{equation*}
From the above argument and by the proof of Proposition \ref{prop2.6}, we can also show that
\begin{equation*}
\sum_{j}|\lambda_j|^q\frac{\int_{D(a_j,r)}|u_k(z)|^qd\mu(z)}{(1-|a_j|^2)^{(\frac{2}{p}+k)q}}\lesssim \left(\sum_{j}|\lambda_j|^p\right)^{q/p},
\end{equation*}
for $k=0,1,\cdots,n.$ This implies that the sequence
\begin{equation*}
\left\{\frac{\int_{D(a_j,r)}|u_k(z)|^qd\mu(z)}{(1-|a_j|^2)^{(\frac{2}{p}+k)q}}\right\}_{j=1}^{\infty}
\end{equation*}
belongs to the dual of $l^{\frac{p}{q}}$ or, equivalently,
\begin{equation*}
\sum_{j}\left(\frac{\int_{D(a_j,r)}|u_k(z)|^qd\mu(z)}{(1-|a_j|^2)^{(\frac{2}{p}+k)q}}\right)^{\frac{p}{p-q}}<\infty.
\end{equation*}

Since $(1-|a_j|^2)^2\simeq |D(a_j,r)|$, we can rewrite the above inequality as
\begin{equation*}
\sum_{j}\left(\frac{\int_{D(a_j,r)}|u_k(z)|^qd\mu(z)}{(1-|a_j|^2)^{kq+2}}\right)^{\frac{p}{p-q}}|D(a_j,r)|<\infty.
\end{equation*}
Note that $D(z,r)\subset D(a_j,2r)$ whenever $z\in D(a_j,r)$. Hence
\begin{equation*}
\frac{\int_{D(z,r)}|u_k(w)|^qd\mu(w)}{(1-|z|^2)^{kq+2}}\lesssim\frac{\int_{D(a_j,2r)}|u_k(z)|^qd\mu(z)}{(1-|a_j|^2)^{kq+2}},
\end{equation*}
whenever $z\in D(a_j,r)$. Finally, we have
\begin{equation*}
\begin{split}
&\quad\int_{\mathbb{D}}\left(\frac{\int_{D(z,r)}|u_k(w)|^qd\mu(w)}{(1-|z|^2)^{kq+2}}\right)^{\frac{p}{p-q}}dA(z)\\
&\leq \sum_{j}\int_{D(a_j,r)}\left(\frac{\int_{D(z,r)}|u_k(w)|^qd\mu(w)}{(1-|z|^2)^{kq+2}}\right)^{\frac{p}{p-q}}dA(z)\\
&\lesssim\sum_{j}\left(\frac{\int_{D(a_j,2r)}|u_k(z)|^qd\mu(z)}{(1-|a_j|^2)^{kq+2}}\right)^{\frac{p}{p-q}}|D(a_j,r)|<\infty.
\end{split}
\end{equation*}
It follows from Lemma \ref{lemma2.4} that $|u_k|^qd\mu$ is a $(k,p,q)$-Carleson measure for each $0\leq k\leq n$. This completes the proof of Theorem 1.2.
\end{proof}

\section{Proof of Theorem 1.3 and Theorem 1.4}

\begin{proof}[\bf Proof of Theorem 1.3]
According to Lemma \ref{lemma2.1}, for any $f\in A^p$,
\begin{equation*}
\begin{split}
\|I_{\mathbf{g}}^{(n)}f\|_q&\simeq\sum_{j=0}^{n-1}|(I_{\mathbf{g}}^{(n)}f)^{(j)}(0)|+\left(\int_{\mathbb{D}}|(I_{\mathbf{g}}^{(n)}f)^{(n)}(z)|^q(1-|z|^2)^{nq}dA(z)\right)^{\frac{1}{q}}\\
&=\left(\int_{\mathbb{D}}\left|\sum_{k=0}^{n-1}f^{(k)}(z)g_k(z)\right|^q(1-|z|^2)^{nq}dA(z)\right)^{\frac{1}{q}}.
\end{split}
\end{equation*}
Thus, $I_{\mathbf{g}}^{(n)}:A^p\to A^q$ is bounded if and only if $(1-|z|^2)^{nq}dA(z)$ is a $(\mathbf{g},p,q)$-Sobolev Carleson measure. By Theorem \ref{Theorem1.1}, this is equivalent to the measures $|g_k|^q(1-|z|^2)^{nq}dA$ being $(k,p,q)$-Carleson for $k=0,1,\cdots, n-1$. By Lemma \ref{lemma2.4}, this is then equivalent to 
\begin{equation}\label{equa4.1}
\sup_{w\in\mathbb{D}}\frac{\int_{D(w,r)}|g_k(z)|^q(1-|z|^2)^{nq}dA(z)}{(1-|w|^2)^{(k+\frac{2}{p})q}}<\infty,\quad k=0,1,\cdots,n-1.
\end{equation}

On one hand, by \cite[Proposition 4.3.8]{Zk1} and \eqref{equa2.3}, we have 
\begin{equation*}
\frac{\int_{D(w,r)}|g_k(z)|^q(1-|z|^2)^{nq}dA(z)}{(1-|w|^2)^{(k+\frac{2}{p})q}}\gtrsim |g_k(w)|^q(1-|w|^2)^{(n-k-\frac{2}{p})q+2}.
\end{equation*}
Therefore, (4.1) implies that
\begin{equation*}
\sup_{z\in\mathbb{D}}|g_k(z)|(1-|z|^2)^{n-k-\frac{2}{p}+\frac{2}{q}}<\infty.
\end{equation*}
i.e., $g_k\in\mathcal{B}^{0,n-k-\frac{2}{p}+\frac{2}{q}}$. In this case, $g_k=0$ when $n-k<\frac{2}{p}-\frac{2}{q}$.

On the other hand, if $M_k:=\sup_{z\in\mathbb{D}}|g_k(z)|(1-|z|^2)^{n-k-\frac{2}{p}+\frac{2}{q}}<\infty$, then using \eqref{equa2.3}, we get
\begin{equation*}
\begin{split}
&\quad\frac{\int_{D(w,r)}|g_k(z)|^q(1-|z|^2)^{nq}dA(z)}{(1-|w|^2)^{(k+\frac{2}{p})q}}\\
&\lesssim \int_{D(w,r)}|g_k(z)|^q(1-|z|^2)^{(n-k-\frac{2}{p})q}dA(z)\\
&\lesssim M_k^q\int_{D(w,r)}\frac{1}{(1-|z|^2)^2}dA(z)\\
&\lesssim M_k^q,
\end{split}
\end{equation*}
for all $w\in\mathbb{D}$ and $k=0,1,\cdots,n-1$. Now the proof of (a) is complete.

Similarly, $I_{\mathbf{g}}^{(n)}:A^p\to A^q$ is compact if and only if $(1-|z|^2)^{nq}dA(z)$ is a vanishing $(\mathbf{g},p,q)$-Sobolev Carleson measure. According to Theorem \ref{Theorem1.1} and Lemma \ref{lemma2.4}, this is equivalent to 
\begin{equation}\label{equa4.2}
\lim_{|w|\to1}\frac{\left(\int_{D(w,r)}|g_k(z)|^q(1-|z|^2)^{nq}dA(z)\right)^{\frac{1}{q}}}{(1-|w|^2)^{\frac{2}{p}+k}}=0
\end{equation}
for every $0\leq k\leq n-1$. Just like the proof of (a), \eqref{equa4.2} is equivalent to 
\begin{equation*}
\lim_{|z|\to 1^{-}}|g_k(z)|(1-|z|^2)^{n-k-\frac{2}{p}+\frac{2}{q}}=0.
\end{equation*}
i.e. $g_k\in\mathcal{B}_{0}^{0,n-k-\frac{2}{p}+\frac{2}{q}}$. And in this case, $g_k=0$ when $n-k\leq \frac{2}{p}-\frac{2}{q}$. The proof is complete.
\end{proof}

\begin{proof}[\bf Proof of Theorem 1.4]
The equivalence of (a) and (b) follows from the fact that $(1-|z|^2)^{nq}dA(z)$ is a $(\mathbf{g},p,q)$-Sobolev Carleson measure if and only if it is a vanishing $(\mathbf{g},p,q)$-Sobolev Carleson measure when $0<q<p<\infty$. We now prove the equivalence of (a) and (c).

First assume $I_{\mathbf{g}}^{(n)}:A^p\to A^q$ is bounded. By Theorem \ref{Theorem1.2} and Lemma \ref{lemma2.4}, we get
\begin{equation}\label{equa4.3}
\frac{\int_{D(z,r)}|g_k(w)|^q(1-|w|^2)^{nq}dA(w)}{(1-|z|^2)^{2+kq}}\in L^{\frac{p}{p-q}}(\mathbb{D},dA),
\end{equation}
for every $0\leq k\leq n-1$. Applying \cite[Proposition 4.3.8]{Zk1} and \ref{equa2.3}, we have
\begin{equation*}
\frac{\int_{D(z,r)}|g_k(w)|^q(1-|w|^2)^{nq}dA(w)}{(1-|z|^2)^{2+kq}}\gtrsim(1-|z|^2)^{(n-k)q}|g_k(z)|^q.
\end{equation*}
Combining this with \eqref{equa4.3}, we get
\begin{equation*}
\int_{\mathbb{D}}\left|g_k(z)(1-|z|^2)^{n-k}\right|^{\frac{pq}{p-q}}dA(z)<\infty.
\end{equation*}
This shows that (a) implies (c).

Conversely, assume the condition (c) holds. Then for any $f\in A^p$, by Lemma \ref{lemma2.1} and H$\rm\ddot{o}$lder's inequality, we have
\begin{equation*}
\begin{split}
\|I_{\mathbf{g}}^{(n)}f\|_q^q&\simeq\int_{\mathbb{D}}\left|\sum_{k=0}^{n-1}f^{(k)}(z)g_k(z)\right|^q(1-|z|^2)^{nq}dA(z)\\
&\lesssim\sum_{k=0}^{n-1}\int_{\mathbb{D}}|f^{(k)}(z)g_k(z)|^q(1-|z|^2)^{nq}dA(z)\\
&\lesssim\sum_{k=0}^{n-1}\left(\int_{\mathbb{D}}|f^{(k)}(z)|^p(1-|z|^2)^{kp}dA(z)\right)^{\frac{q}{p}}\left(\int_{\mathbb{D}}|g_k(z)(1-|z|^2)^{n-k}|^{\frac{pq}{p-q}}dA(z)\right)^{\frac{p-q}{p}}\\
&\lesssim \sum_{k=0}^{n-1}\|f\|_p^q\left(\int_{\mathbb{D}}|g_k(z)(1-|z|^2)^{n-k}|^{\frac{pq}{p-q}}dA(z)\right)^{\frac{p-q}{p}}\\
&\lesssim \|f\|_p^q.
\end{split}
\end{equation*}
This establishes the boundedness of $I_{\mathbf{g}}^{(n)}$ from $A^p$ to $A^q$. 
\end{proof}

\begin{proof}[\bf Proof of Corollary 1.5]
Let $g_k=g^{n-k}$ in Theorem \ref{Theorem1.3}, we know that $I_{g,a}:A^p\to A^q$ if and only if 
\begin{equation*}
\sup_{z\in\mathbb{D}}|g^{n-k}(z)|(1-|z|^2)^{n-k-\frac{2}{p}+\frac{2}{q}}<\infty,
\end{equation*}
for each $k=0,1,\cdots,n-1$.
These are exactly equivalent to that 
\begin{equation*}
\sup_{z\in\mathbb{D}}|g'(z)|(1-|z|^2)^{1-\frac{2}{p}+\frac{2}{q}}<\infty.
\end{equation*}
i.e. $g\in \mathcal{B}^{1,1-\frac{2}{p}+\frac{2}{q}}$. The proof of the compactness part is similar and we omit the routine details.
\end{proof}

\begin{proof}[\bf Proof of Corollary 1.6]
Let $g_k=g^{(n-k)}$ in Theorem \ref{Theorem1.4}. For $0<q<p<\infty$, we know that $I_{g,a}:A^p\to A^q$ is bounded if and only if it is compact, and if and only if 
$$\int_{\mathbb{D}}|g^{(n-k)}(z)(1-|z|^2)^{n-k}|^{\frac{pq}{p-q}}dA(z)<\infty$$ 
for every $0\leq k\leq n-1$, which is then equivalent to that $g\in A^{\frac{pq}{p-q}}$ by Lemma \ref{lemma2.1}. This gives the desired result.
\end{proof}

\section{Proof of Theorem 1.7}

\begin{proof}
We show that Theorem 1.7 is a consequence of Theorem 1.3. To this end, we recall the operator 
$$I_{\mathbf{g}}^{(n)}f=I^n(f^{(n-1)}g_{n-1}+\cdots+fg_0).$$
By Theorem \ref{Theorem1.3}, there exists a constant $C>0$, depending only on $p$, such that 
$$\|I_{\mathbf{g}}^{(n)}\|\leq C\sum_{k=0}^{n-1}\|g_k\|_{0,n-k}.$$
If $\|g_k\|_{0,n-k}\leq \frac{1}{nC}$ for all $0\leq k\leq n-1$, then the norm $\|I_{\mathbf{g}}^{(n)}\|<1$, hence $-1$ is in the resolvent of $I_{\mathbf{g}}^{(n)}$ on $A^p$.

Let $f$ be any solution of the differential equation \eqref{equa1.1}. There exists $F_0\in A^p$ such that $F_0-I^nF$ is a polynomial of degree less than $n$ and $F_0^{(i)}(0)=f^{(i)}(0)$ for $0\leq i\leq n$. Choose $\tilde{f}\in A^p$ such that $I_{\mathbf{g}}^{(n)}\tilde{f}+\tilde{f}=F_0$. But $\tilde{f}$ and $f$ satisfy the same differential equation with the same initial condition. By the uniqueness of the solution of the initial value problem, $f=\tilde{f}\in A^p$.

If $g_k\in \mathcal{B}_{0}^{0,n-k}$, then $I_{\mathbf{g}}^{(n)}$ is compact on $A^p$. In this case, we will prove that the spectrum of $I_{\mathbf{g}}^{(n)}$ is the singleton $\{0\}$. Then the desired result will follow from the previous argument. Thus it suffices to prove that $I_{\mathbf{g}}^{(n)}$ has no nonzero eigenvalue. To see this, note that if $\lambda\in\mathbb{C}\backslash\{0\}$, the equation $I_{\mathbf{g}}^{(n)}f=\lambda f$ is equivalent to the initial value problem
\begin{equation*}
f^{(n-1)}g_{n-1}+\cdots+fg_0=\lambda f^{(n)},\quad f(0)=\cdots=f^{(n-1)}(0)=0.
\end{equation*}
By the uniqueness of solutions, we must have $f=0$.
\end{proof}

\section{Proof of Theorem 1.8 and Theorem 1.9}

In this section, we present the proof of Theorems 1.8 and 1.9 about the rigidity of the sum operator $L_{\mathbf{u},\varphi}^{(n)}$.

Assume $u\in H(\mathbb{D})$ and $\varphi$ is an analytic self-map of $\mathbb{D}$. We define the weighted pull-back measure $\mu_{u,\varphi}$ on $\mathbb{D}$ by
\begin{equation*}
\mu_{u,\varphi}(E)=\int_{\varphi^{-1}(E)}|u(z)|^qdA(z),
\end{equation*}
where $E$ is any Borel subset of $\mathbb{D}$. For any $f\in H(\mathbb{D})$, 
\begin{equation*}
\begin{split}
\|W_{u,\varphi}^{(k)}f\|_q^q&=\int_{\mathbb{D}}|u(z)f^{(k)}(\varphi(z))|^qdA(z)\\
&=\int_{\mathbb{D}}|f^{(k)}(w)|^qd\mu_{u,\varphi}(w).
\end{split}
\end{equation*}
Thus, for $0<p,q<\infty$, $W_{u,\varphi}^{(k)}:A^p\to A^q$ is bounded (compact) if and only if $\mu_{u,\varphi}$ is (vanishing) $(k,p,q)$-Carleson measure. For simplicity, set $\mu_k:=\mu_{u_k,\varphi}$ for $0\leq k\leq n$.

We are now ready to prove Theorem 1.8.
\begin{proof}[\bf Proof of Theorem 1.8]
The ``if part" is trivial, we only need to prove the ``only if part". Through this proof, we study the cases $p\leq q$ and $p>q$ separately.

{\bf The case $\mathbf {p\leq q}$.} The case $n=0$ is trivial. Now we prove the case $n=1$. If $W_{u_0,\varphi}+W_{u_1,\varphi}^{(1)}:A^p\to A^q$ is bounded, then
\begin{equation}\label{equa6.1}
\sup_{w\in\mathbb{D}}\int_{\mathbb{D}}|u_0(z)k_w^{[0]}(\varphi(z))+u_1(z)(k_w^{[0]})'(\varphi(z))|^qdA(z)<\infty
\end{equation}
and 
\begin{equation}\label{equa6.2}
\sup_{w\in\mathbb{D}}\int_{\mathbb{D}}|u_0(z)k_w^{[1]}(\varphi(z))+u_1(z)(k_w^{[1]})'(\varphi(z))|^qdA(z)<\infty,
\end{equation}
where $k_w^{[i]}(z)=(1-|w|^2)^{\gamma+i-\frac{2}{p}}K_w^{[i]}(z)$, $i=0,1$, as in the proof of Theorem 1.1.

Fix $r>0$, the estimate in \eqref{equa6.1} clear gives
\begin{equation*}
\sup_{w\in\mathbb{D}}(1-|w|^2)^{\gamma-\frac{2}{p}}\int_{\varphi^{-1}(D(w,r))}\left|\frac{u_0(z)}{(1-\overline{w}\varphi(z))^{\gamma}}+\frac{u_1(z)\gamma\overline{w}}{(1-\overline{w}\varphi(z))^{\gamma+1}}\right|^qdA(z)<\infty.
\end{equation*}
Recall from \eqref{equa2.3} that $|1-\overline{w}\varphi(z)|\simeq (1-|w|^2)$ for $z\in\varphi^{-1}(D(w,r))$. Thus
\begin{equation}\label{equa6.3}
\sup_{w\in\mathbb{D}}\frac{1}{(1-|w|^2)^{\frac{2q}{p}}}\int_{\varphi^{-1}(D(w,r))}\left|u_0(z)+\frac{u_1(z)\gamma\overline{w}}{1-\overline{w}\varphi(z)}\right|^qdA(z)<\infty.
\end{equation}
Similarly, it follows from \eqref{equa6.2} that
\begin{equation}\label{equa6.4}
\sup_{w\in\mathbb{D}}\frac{1}{(1-|w|^2)^{\frac{2q}{p}}}\int_{\varphi^{-1}(D(w,r))}\left|u_0(z)\varphi(z)+\frac{u_1(z)(1+\gamma\overline{w}\varphi(z))}{1-\overline{w}\varphi(z)}\right|^qdA(z)<\infty.
\end{equation}
Let $|\varphi|\cdot$(6.3)+(6.4), then by triangle inequality and \eqref{equa2.3}, we get
\begin{equation*}
\sup_{w\in\mathbb{D}}\frac{1}{(1-|w|^2)^{(\frac{2}{p}+1)q}}\int_{\varphi^{-1}(D(w,r))}|u_1(z)|^qdA(z)<\infty.
\end{equation*}
This shows that $\mu_1$ is a $(1,p,q)$-Carleson measure. Thus $W_{u_1,\varphi}^{(1)}:A^p\to A^q$ is bounded, then so is $W_{u_0,\varphi}$. Modifying the proof as in Proposition 2.6, the general case $n\geq 2$ can be also proved by induction. And the ``compactness part" is similar, we omit the routine details.

{\bf The case $\mathbf {p>q}$.} In this case, by Lemma \ref{lemma2.4}, we know that $\mu_k$ is a $(k,p,q)$-Carleson measure if and only if it is a vanishing $(k,p,q)$-Carleson measure. And then $W_{u_k,\varphi}^{(k)}:A^p\to A^q$ is bounded if and only if it is compact. So we only need to show that each $\mu_k$, $0\leq k\leq n$, is a $(k,p,q)$-Carleson measure if $L_{\mathbf{u},\varphi}^{(n)}:A^p\to A^q$ is bounded. To this end, set 
\begin{equation*}
h_k(z)=\sum_{j}\lambda_j\frac{z^k(1-|a_j|^2)^{\gamma+k-\frac{2}{p}}}{(1-\overline{a}_jz)^{\gamma+k}},\quad k=0,1,\cdots,n,
\end{equation*}
as in the proof of Theorem 1.2, where $\{\lambda_j\}\in l^p$ and $\{a_j\}$ is any $r$-lattice in the Bergman metric.

If $L_{\mathbf{u},\varphi}^{(n)}:A^p\to A^q$ is bounded, then
\begin{equation*}
\int_{\mathbb{D}}\left|\sum_{i=0}^nu_i(z)h_k^{(i)}(\varphi(z))\right|^qdA(z)\lesssim \left(\sum_{j}|\lambda_j|^p\right)^{\frac{q}{p}}
\end{equation*}
for $k=0,1,\cdots,n$. Just like the proof in Theorem 1.2 and Proposition 2.6, we could obtain
\begin{equation*}
\sum_{j}|\lambda_j|^q\frac{\mu_{k}((D_{a_j},r))}{(1-|a_j|^2)^{(\frac{2}{p}+k)q}}\lesssim\left(\sum_{j}|\lambda_j|^p\right)^{\frac{q}{p}}
\end{equation*}
for every $k=0,1,\cdots,n$. And then
\begin{equation*}
\sum_{j}\left[\frac{\mu_k(D(a_j,r))}{(1-|a_j|^2)^{(\frac{2}{p}+k)q}}\right]^{\frac{p}{p-q}}<\infty.
\end{equation*}
Therefore, by \eqref{equa2.3}, we have
\begin{equation*}
\begin{split}
&\quad\int_{\mathbb{D}}\left[\frac{\mu_k(D(z,r))}{(1-|z|^2)^{kq+2}}\right]^{\frac{p}{p-q}}dA(z)\\
&\leq \sum_{j}\int_{D(a_j,r)}\left[\frac{\mu_k(D(z,r))}{(1-|z|^2)^{kq+2}}\right]^{\frac{p}{p-q}}dA(z)\\
&\lesssim \sum_{j}\left[\frac{\mu_k(D(a_j,2r))}{(1-|a_j|^2)^{kq+2}}\right]^{\frac{p}{p-q}}|D(a_j,r)|\\
&\simeq\sum_{j}\left[\frac{\mu_k(D(a_j,2r))}{(1-|a_j|^2)^{(\frac{2}{p}+k)q}}\right]^{\frac{p}{p-q}}<\infty.
\end{split}
\end{equation*}
Then Lemma \ref{lemma2.4} tells us that $\mu_k$ is a $(k,p,q)$-Carleson measure for every $0\leq k\leq n$. The proof is complete.
\end{proof}

We now proceed to the proof of Theorem 1.9. Before that, we need to know the action of the ajoint of $L_{\mathbf{u},\varphi}^{(n)}$ on the kernel functions.

\begin{lemma}\label{lemma6.1}
Suppose $L_{\mathbf{u},\varphi}^{(n)}$ is bounded on $A^2$, then
\begin{equation*}
(L_{\mathbf{u},\varphi}^{(n)})^{*}\mathcal{K}_z=\sum_{j=0}^n\overline{u_j(z)}\mathcal{K}_{\varphi(z)}^{[j]},
\end{equation*}
for all $z\in\mathbb{D}$.
\end{lemma}

\begin{proof}
For any $f\in A^2$ and $z\in\mathbb{D}$, we have
\begin{equation*}
\begin{split}
\langle f,(L_{\mathbf{u},\varphi}^{(n)})^*\mathcal{K}_z\rangle_2&=\langle L_{\mathbf{u},\varphi}^{(n)}f,\mathcal{K}_z\rangle_2\\
&=\sum_{j=0}^nu_j(z)f^{(j)}(\varphi(z))\\
&=\langle f,\sum_{j=0}^n\overline{u_j(z)}\mathcal{K}_{\varphi(z)}^{[j]}\rangle_2.
\end{split}
\end{equation*}
Thus,
\begin{equation*}
(L_{\mathbf{u},\varphi}^{(n)})^{*}\mathcal{K}_z=\sum_{j=0}^n\overline{u_j(z)}\mathcal{K}_{\varphi(z)}^{[j]}.
\end{equation*}
\end{proof}

\begin{proof}[\bf Proof of Theorem 1.9]
The ``if part" is trivial, we only need to prove the ``only if part".

Suppose $L_{\mathbf{u},\varphi}^{(n)}$ is Hilbert-Schmidt on $A^2$, then $L_{\mathbf{u},\varphi}^{(n)}(L_{\mathbf{u},\varphi}^{(n)})^*$ belongs to the trace class. According to \cite[Theorem6.6]{Zk1},
$$\langle L_{\mathbf{u},\varphi}^{(n)}(L_{\mathbf{u},\varphi}^{(n)})^*\mathcal{K}_z,\mathcal{K}_z\rangle_2\in L^1(\mathbb{D},dA).$$
On the other hand, by Lemma \ref{lemma6.1} and Proposition \ref{prop2.6}, we have
\begin{equation*}
\begin{split}
\langle L_{\mathbf{u},\varphi}^{(n)}(L_{\mathbf{u},\varphi}^{(n)})^*\mathcal{K}_z,\mathcal{K}_z\rangle_2&=\left\|(L_{\mathbf{u},\varphi}^{(n)})^*\mathcal{K}_z\right\|_2^2\\
&=\left\|\sum_{j=0}^n\overline{u_j(z)}\mathcal{K}_{\varphi(z)}^{[j]}\right\|_2^2\\
&\simeq\sum_{j=0}^{n}|u_j(z)|^2\|\mathcal{K}_{\varphi(z)}^{[j]}\|_2^2.
\end{split}
\end{equation*}
Therefore, by Lemma \ref{lemma2.3}, 
\begin{equation*}
\int_{\mathbb{D}}|u_j(z)|^2\|\mathcal{K}_{\varphi(z)}^{[j]}\|_2^2dA(z)\simeq\int_{\mathbb{D}}\frac{|u_j(z)|^2}{(1-|\varphi(z)|^2)^{2+2j}}dA(z)<\infty.
\end{equation*}
Then, by \cite[Theorem3.1]{Zx1}, $W_{u_j,\varphi}^{(j)}$ is Hilbert-Schmidt on $A^(\mathbb{D})$ for every $j=0,1,\cdots,n$.
\end{proof}

\begin{remark}
Note that the differentiation-composition operator $D_{\varphi}^{(n)}$ can be represented as a combination of generalized weighted composition operators with different orders by Fa\`a De Bruno's formula concerning higher order derivatives of a composite function. Thus, the boundedness and compactness of $D_{\varphi}^{(n)}$ can be characterized via the corresponding property for $L_{\mathbf{u},\varphi}^{(n)}$.
\end{remark}

\begin{remark}
For $0<p<\infty$ and $\alpha>-1$, the classical weighted Bergman space $A_{\alpha}^p$ is defined by 
\begin{equation*}
A_{\alpha}^p=\left\{f\in H(\mathbb{D}): \|f\|_{\alpha,p}^p=(\alpha+1)\int_{\mathbb{D}}|f(z)|^p(1-|z|^2)^{\alpha}dA(z)<\infty\right\}.
\end{equation*}
The results can be extended to $A_{\alpha}^p$ without much difficulty. The Dirichlet-type space $D^p_{\alpha}$ is the space consisting of all $f\in H(\mathbb{D})$ such that $f'\in A_{\alpha}^p$ and $\|f\|_{D_{\alpha}^p}=|f(0)|+\|f'\|_{\alpha,p}$. For quite some time, the characterizations for boundedness and compactness of $W_{u,\varphi}$ on $D_{\alpha}^p$ are incomplete. Inspired by the key idea in Proposition 2.6, we can draw the following result: $W_{u,\varphi}:D_{\alpha}^p\to D_{\alpha}^q$ is bounded (compact, resp.) if and only if both $W_{u\varphi',\varphi}: A_{\alpha}^p\to A_{\alpha}^q$ and $W_{u',\varphi}:D_{\alpha}^p\to A_{\alpha}^q$ are bounded (compact, resp.). In particular, let $\varphi(z)=z$, then $u$ is a multiplier of $D_{\alpha}^p$ if and only if $u$ is a multiplier of $A_{\alpha}^p$, i.e., $u\in H^{\infty}$ and $M_{u'}$ is bounded from $D_{\alpha}^p$ to $A_{\alpha}^p$, which coinsides with the result in \cite[Lemma4.1]{Wz}.
\end{remark}


\begin{thebibliography}{99}

\bibitem{AaSa} A. Aleman and A. Siskakis,
               {\it An integral operator on $H^p$},
               Complex Variables Theory Appl. {\bf 28} (1995), 149-158.

\bibitem{AaSa1} A. Aleman and A. Siskakis,
                {\it Integration operators on Bergman spaces},
                Indiana Univ. Math. J. {\bf 46} (1997), 337-356.

\bibitem{Ae} E. Amar,
             {\it Suites d'interpolation pour les class de Bergman de la boule et du polydisque de $\mathbb{C}^n$},
             Canad. J. Math. {\bf 30} (1978), 711-737.

\bibitem{Ah} H. Arroussi,
             {\it Weighted composition operators on Bergman space $A_{\omega}^p$},
             Math. Nachr. {\bf 295} (2022), 631-656.

\bibitem{Cn} N. Chalmoukis,
             {\it Generalized integration operators on Hardy spaces},
              Proc. Amer. Math. Soc. {\bf 148} (2020), 3325-3337.

\bibitem{CbCkKhYj} B. Choe, K. Choi, H. Koo and J. Yang,
                   {\it Difference of weighted composition operators},
                    J. Funct. Anal. {\bf 278} (2020), no. 5, 108401, 38pp.

\bibitem{Co} O. Constantion,
             {\it A Volterra-type integration operator on Fock spaces.} 
             Proc. Amer. Math. Soc. {\bf 140} (2010), 4247-4257.

\bibitem{CcMb} C. Cowen and B. MacCluer,
               {\it Composition Operators on Spaces of Analytic Functions},
               New York: CRC Press (1995).

\bibitem{CzZr} $\rm \breve{Z}$. $\rm \breve{C}$u$\rm \breve{c}$kovi\'c and R.Zhao,
               {\it Weighted composition operators on the Bergman spaces},
               J. London Math. Soc. {\bf 70} (2004), 499-511.

\bibitem{CzZr1} $\rm \breve{Z}$. $\rm \breve{C}$u$\rm \breve{c}$kovi\'c and R.Zhao,
               {\it Weighted composition operators between different weighted Bergman spaces and different Hardy spaces},
                Illinois J. Math. {\bf 51} (2007), 479-498.

\bibitem{Dp} P. Duren,
             {\it Theory of $H^p$ Spaces},
             Academic Press, New York-London, 1970. Reprint: Dover, Mineola, New York, 2000.

\bibitem{Ff} F. Forelli,
             {\it The isometries of $H^p$},
             Canad. J. Math. {\bf 16} (1964), 721-728.

\bibitem{GpGdPj} P. Galanopoulos, D. Girela and J. Pal\'ez,
                 {\it Multipliers and integration operators on Dirichlet spaces},
                 Trans. Amer. Math. Soc. {\bf 363} (2011), 1855-1886.

\bibitem{GjHjRj} J. Gr$\rm\ddot{o}$hn, J. Huusko and J. R$\rm\ddot{a}$tty$\rm\ddot{a}$,
                 {\it Linear differential equations with slowly growing solutions},
                 Trans. Amer. Math. Soc. {\bf 370} (2018), 7201-7227.

\bibitem{Hw} W. Hastings,
             {\it A Carleson measure theorem for Bergman spaces},
              Proc. Amer. Math. Soc. {\bf 52} (1975), 237-241.

\bibitem{LbOYc} B. Li and C. Ouyang,
                {\it Higher radial derivative of Bloch type functions},
                Acta Math. Sci. {\bf 22} (2022), 433-445.

\bibitem{LhWl} H. Li and H. Wulan,
               {\it Linear differential equations with solutions in the $Q_k$ spaces},
               J. Math. Anal. Appl. {\bf 375} (2011), 478-489.

\bibitem{Ld} D. Luecking,
             {\it A technique for characterizing Carleson measures on Bergman spaces},
             Proc. Amer. Math. Soc. {\bf 87} (1983), 656-660.

\bibitem{Ld1} D. Luecking,
              {\it Forward and reverse Carleson inequalities for functions in Bergman spaces and their derivatives},
               Amer. J. Math. {\bf 107} (1985), 85-111.

\bibitem{Ld2} D. Luecking,
              {\it Embedding theorem for spaces of analytic functions via Khinchine's inequality},
               Michigan Math. J. {\bf 40} (1993), 333-358.

\bibitem{Kc} C. Kolaski,
             {\it Isometries of weighted Bergman spaces},
             Canad. J. Math. {\bf 34} (1982), 910-915.

\bibitem{Pc} C. Pommerenke,
             {\it On the mean growth of the solutions of complex linear differential equations in the disk},
             Complex Variables, {\bf 1} (1982), 23-38.

\bibitem{PjRj} J. Pel\'aez and J. R$\rm\ddot{a}$tty$\rm\ddot{a}$,
               {\it Weighted Bergman spaces induced by rapidly increasing weights},
               Mem. Amer. Math. Soc. {\bf 227} (2014), 1066.

\bibitem{Rj} J. R$\rm\ddot{a}$tty$\rm\ddot{a}$,
             {\it Linear differential equations with solutions in the Hardy spaces},
             Complex Var. Elliptic Equ. {\bf 52} (2007), 785-795.

\bibitem{Sj} J. Shapiro,
             {\it Composition operators and Classical Function Theory},
             New York: Springer Verlag. (1993).

\bibitem{SsSaBa} S. Stevi\'c, A. Sharma and A. Bhat,
                 {\it Products of multiplication composition and differentiation operators on weighted Bergman spaces},
                 Appl. Math. Comput. {\bf 217} (2011), 8115-8125.

\bibitem{SyLjHg} Y. Sun, J. Liu and G. Hu,
                 {\it Complex linear differential equations with solutions in the $H_K^2$ spaces},
                 Complex Var. Elliptic Equ. {\bf 67} (2022), 2577-2588.

\bibitem{WsWmGx} S. Wang, M. Wang and X. Guo,
                 {\it Difference of Stevi\c-Sharma operators},
                 Banach J. Math. Anal. {\bf 14} (2020), 1019-1054.

\bibitem{Wz} Z. Wu,
             {\it Carleson measures and multipliers for Dirichlet spaces},
             J. Funct. Anal. {\bf 169} (1999), 148-163.

\bibitem{Xj} J. Xiao,
             {\it Riemann-Stiejes operators on weighted Bloch and Bergman spaces of the unit ball},
             J. London Math Soc. {\bf 70} (2004), 199-214.

\bibitem{YyLy} Y. Yu and Y. Liu,
               {\it On Stevi\'c-type operator from $H^{\infty}$ space to the logarithmic Bloch spaces},
               Complex Anal. Oper. Theory. {\bf 9} (2015), 1759-1780.

\bibitem{ZfLy} F. Zhang and Y. Liu,
               {\it On a Stevi\/c-Sharma operator from Hardy spaces to Zygmund type spaces on the unit disk},
               Complex Anal. Oper. Theory. {\bf 12} (2018), 81-100.

\bibitem{ZrZk} R. Zhao and K. Zhu,
               {\it Theory of Bergman spaces in the unit ball of $\mathbb{C}^n$},
               Mem. Soc. Math. France {\bf 115} (2008), 103 pages.

\bibitem{Zk} K. Zhu,
             {\it Spaces of Holomorphic Functions in the Unit Ball},
              Springer, 2005.

\bibitem{Zk1} K. Zhu,
              {\it Operator Theory in Function Spaces},
               American Mathematical Society, 2007.

\bibitem{Zk2} K. Zhu,
              {\it Bloch type spaces of analytic functions},
              Rocky Mountain J. Math. {\bf 23} (1993), 1143-1177.

\bibitem{Zx} X. Zhu,
             {\it Generalized weighted composition operators on weighted Bergman spaces},
             Numer. Funct. Anal. Optim. {\bf 30} (2009), 881-893.

\bibitem{Zx1} X. Zhu,
              {\it Generalized weighted composition operators on weighted Bergman spaces II},
              Math. Inequal. Appl. {\bf 22} (2019), 1055-1066.
\end{thebibliography}
\end{document}